\documentclass[12pt,reqno]{amsart}
\usepackage{amsopn,amssymb,mathrsfs,mathrsfs,a4wide,color,url}

\newtheorem{thm}{Theorem}[section]
\newtheorem{cor}[thm]{Corollary}
\newtheorem{lem}[thm]{Lemma}
\newtheorem{prop}[thm]{Proposition}

\theoremstyle{definition}
\newtheorem{defn}[thm]{Definition}

\newtheorem{rem}[thm]{Remark}

\newcommand{\bdy}{\partial}                                                 
\newcommand{\ep}{\varepsilon}
\newcommand{\free}[1]{\ensuremath{\mathcal{F}({#1})}}                                              
\newcommand{\ind}[1]{\ensuremath{\mbox{\boldmath{$1$}}_{#1}}}                               
\newcommand{\lint}[4]{\ensuremath{\int_{#1}^{#2}{#3}\:\mathrm{d}{#4}}}  
\newcommand{\Lipz}[1]{\ensuremath{\Lip_0({#1})}}
\newcommand{\lspan}[1]{\ensuremath{\aspan({#1})}}                                               
\newcommand{\map}[3]{\ensuremath{{#1}:{#2}\longrightarrow{#3}}}             
\newcommand{\N}{\mathbb{N}}
\newcommand{\Z}{\mathbb{Z}}                                                  
\newcommand{\n}[1]{\ensuremath{\left\|{#1}\right\|}}                        
\newcommand{\ndot}{\ensuremath{\left\|\cdot\right\|}}                                       
\newcommand{\pn}[2]{\ensuremath{\left\|{#1}\right\|_{#2}}}                          
\newcommand{\pndot}[1]{\ensuremath{\left\|\cdot\right\|_{#1}}}                  
\newcommand{\R}{\mathbb{R}}                                                 
\newcommand{\res}[1]{\ensuremath{\!\!\upharpoonright_{#1}}}                 
\newcommand{\set}[2]{\ensuremath{\left\{{#1}\;:\;\,{#2}\right\}}}         
\newcommand{\tri}{{\displaystyle |\kern-.9pt|\kern-.9pt|}}
\newcommand{\tn}[1]{\ensuremath{\tri{#1}\tri}}                                                  
\newcommand{\tndot}{\ensuremath{\tri\cdot\tri}}                                                 
\newcommand{\ts}{\textstyle}

\DeclareMathOperator{\aspan}{span}                                        
\DeclareMathOperator{\diam}{diam}                                         
\DeclareMathOperator{\Lip}{Lip}                                                                                     
\DeclareMathOperator{\intr}{int}                                          
\DeclareMathOperator{\supp}{supp}                                         
\DeclareMathOperator{\conv}{conv}                                          

\allowdisplaybreaks

\begin{document}
\title[The MAP and Lipschitz-free spaces over subsets of $\R^N$]{The metric approximation property and Lipschitz-free spaces over subsets of $\R^N$}
\begin{abstract}
We prove that for certain subsets $M \subseteq \R^N$, $N \geqslant 1$, the Lipschitz-free space
$\free{M}$ has the metric approximation property (MAP), with respect to any norm on $\R^N$. In particular,
$\free{M}$ has the MAP whenever $M$ is a finite-dimensional compact convex set. This should be compared
with a recent result of Godefroy and Ozawa, who showed that there exists a compact
convex subset $M$ of a separable Banach space, for which $\free{M}$ fails the approximation property.
\end{abstract}

\author[E. Perneck\'a]{Eva Perneck\'a}
\address{Department of Mathematical Analysis, Faculty of Mathematics and Physics, Charles University, Sokolovsk\'a 83, 186 75 Praha 8, Czech Republic}
\address{Institute of mathematics of the Academy of Sciences of the Czech Republic, \v Zitn\'a 25, 115 67 Praha 1, Czech Republic}
\address{Universit\'e de Franche-Comt\'e, Laboratoire de Math\'ematiques UMR 6623, 16 route de Gray, 25030 Besan\c con Cedex, France}
\email{pernecka@karlin.mff.cuni.cz}

\author[R. J. Smith]{Richard J. Smith}
\address{School of Mathematical Sciences, University College Dublin, Belfield, Dublin 4, Ireland}
\email{richard.smith@maths.ucd.ie}

\thanks{The authors would like to thank G.\ Godefroy and M.\ Ghergu for making some helpful comments
during the preparation of this work.}
\thanks{The first author was supported in part by GA\v CR P201/11/0345, RVO: 67985840 and SVV-2014-260106.}

\subjclass[2010]{Primary 46B20, 46B28}
\date{\today}
\maketitle

\section{Introduction and main results}\label{s:intro}

Given a metric space $(M,d)$, a function $\map{f}{M}{\R}$ is {\em Lipschitz} if
\[
\Lip(f) := \sup\set{\frac{|f(x)-f(y)|}{d(x,y)}}{x \neq y},
\]
is finite. Evidently, the {\em Lipschitz constant} $\Lip(f)$ of $f$ depends on $d$. A
{\em pointed} metric space is simply a metric space with a distinguished point $x_0$. Given
such a metric space $(M,d)$, we denote by $\Lipz{M}$ the set of all Lipschitz functions as
above which also vanish at $x_0$. This set becomes a Banach space when endowed with the norm
defined by $\n{f}=\Lip(f)$. Let $\delta(x)$, $x \in M$, denote the evaluation functional on
$\Lipz{M}$ given by $\langle f,\delta(x)\rangle=f(x)$ for $f\in\Lipz{M}$. The {\em Lipschitz-free
space} (or simply {\em free space}) $\mathcal{F}(M)$ is defined to be the norm-closed linear
span of $\set{\delta(x)}{x \in M} \subseteq \Lipz{M}^*$. The Dirac map $\delta$ is an isometric
embedding of $M$ into $\free{M}$. If $M$ is a Banach space, then $\delta$ is non-linear with
a linear left inverse given by a barycentre map (see \cite[Proposition 2.1 and Lemma 2.4]{gk:03}).
It turns out that the dual space $\mathcal{F}(M)^*$ of $\free{M}$ is linearly isometric to
$\Lipz{M}$. Moreover, on bounded subsets of $\Lipz{M}$, the weak$^*$-topology induced by $\free{M}$
agrees with the topology of pointwise convergence. The precise location of the distinguished
point $x_0$ within $M$ is not particularly important for us:\ given two distinguished points
$x_0$ and $x_1$, the map $f\mapsto f-f(x_1)\ind{M}$ is a dual linear isometry between the
corresponding spaces of Lipschitz functions. Spaces of Lipschitz functions and their preduals
(referred to as Arens-Eells spaces) are studied in the book \cite{w:99} by Weaver. An introduction
to the theory of Lipschitz-free spaces as defined here can be found in the seminal work of
Godefroy and Kalton \cite{gk:03}.

Lipschitz-free spaces are related to the study of Lipschitz isomorphism classes of spaces. Indeed,
they constitute a tool for abstract linearization of Lipschitz maps
in the following sense. If we use the above Dirac map $\delta$ to identify metric spaces $M$
and $N$ with subsets of the corresponding Lipschitz-free spaces $\free{M}$ and $\free{N}$,
respectively, then any Lipschitz map $L$ from the metric space $M$ into the metric
space $N$ has an extension to a continuous linear map $\hat L$ from $\free{M}$ into $\free{N}$ which preserves the
Lipschitz constant (see \cite{w:99} or \cite[Lemma 2.2]{gk:03}). Moreover, if $N$ is a Banach space, then by composing $\hat L$ with the barycentre map we obtain an extension of $L$ to a continuous linear map from $\free{M}$ into $N$. So Lipschitz-free spaces
can be employed to transfer non-linear problems to a linear setting. Moreover, in \cite[Theorem 3.1]{gk:03},
Godefroy and Kalton establish the so-called isometric lifting property for separable Banach spaces. As stated
in \cite[Corollary 3.3]{gk:03}, this implies that if a separable Banach space $X$ is isometric to a subset
of a Banach space $Y$, then $X$ is already linearly isometric to a subspace of $Y$. This assertion fails
in the non-separable case:\ if $X$ is non-separable and weakly compactly generated, then $X$ does not
embed linearly into $\free{X}$ (see \cite[Section 4]{gk:03}).

Despite their straightforward definition, the linear structure of Lipschitz-free spaces is
relatively difficult to analyse and has not been thoroughly described yet. Elucidating the
properties of the class of Lipschitz-free spaces has been the topic of recent research and
several interesting results have been obtained. The linear isometry between $\Lipz{\R}$ and
$L_{\infty}$, furnished by differentiability almost everywhere, yields a predual linear isometry
between $\free{\R}$ and $L_1$. On the other hand, $\free{\R^2}$ is not linearly isomorphic
to any subspace of $L_1$, as Naor and Schechtman showed in \cite{ns:07} by the discretization
of an argument due to Kislyakov \cite{ki:75}. Finally, the metric spaces whose Lipschitz-free
space is linearly isometric to a subspace of $L_1$ were characterized by Godard in \cite[Theorem 4.2]{g:10}
as metric spaces isometrically embeddable into an $\R$-tree. In \cite[Theorem 3.4]{d:14-proper},
Dalet proved that the Lipschitz-free space over a proper ultrametric space is linearly isometric
to the dual of a space which is linearly isomorphic to $c_0$. Recently, C\'uth and Doucha
managed to relax the assumption and show that the Lipschitz-free space over a separable
ultrametric space is linearly isomorphic to $\ell_1$ \cite[Theorem 2]{cd:14}. One of the main
results of the recent work of Kaufmann states that the Lipschitz-free space $\free{X}$ over a
Banach space $X$ is linearly isomorphic to $\left(\sum_{n=1}^\infty\free{X}\right)_{\ell_1}$
\cite[Theorem 3.1]{kauf:14}. This yields an analogue of Pe\l czy\'nski's decomposition method
\cite[Corollary 3.2]{kauf:14} and enables us to find a class of separable metric spaces whose
Lipschitz-free spaces are linearly isomorphic to $\free{c_0}$ \cite[Corollary 3.4]{kauf:14}.
This class contains in particular all $C(K)$ spaces where $K$ is an infinite compact
metric space, thus this result provides another way of obtaining examples, first exhibited by Dutrieux
and Ferenczi, of non-Lipschitz isomorphic Banach spaces having linearly isomorphic Lipschitz-free
spaces \cite[Theorem 5]{df:05}.

In this note we concentrate on approximation properties enjoyed by certain Lipschitz-free spaces.
Recall that a Banach space $X$ has the {\em approximation property} (AP), or the {\em $\lambda$-bounded
approximation property} ($\lambda$-BAP), if the identity operator on $X$ lies in the closure of
the set of bounded, or uniformly $\lambda$-bounded, finite-rank operators on $X$, respectively,
where closure is taken with respect to the topology of uniform convergence on norm-compact subsets
of $X$. If $\lambda$ can be taken to be unity then $X$ is said to have the {\em metric approximation
property} (MAP). A Banach space has the {\em bounded approximation property} (BAP) if it has the
$\lambda$-BAP for some $\lambda$. In the case of the BAP the closure above can be taken with respect
to the strong operator topology.

Godefroy and Kalton in \cite[Theorem 5.3]{gk:03} proved that a Banach space $X$ has the $\lambda$-BAP
if and only if $\free{X}$ has the $\lambda$-BAP. In view of the aforementioned linearization of
Lipschitz maps via Lipschitz-free spaces \cite[Lemma 2.2]{gk:03}, it follows that the BAP is stable
under Lipschitz isomorphisms between Banach spaces. By Godefroy and Ozawa \cite[Theorem 4]{go:14},
every separable Banach space $X$ is linearly isometric to a $1$-complemented subspace of $\free{K}$,
where $K\subseteq X$ is closed, convex and generates $X$. Applying this result to a separable Banach
space failing the AP constructed by Enflo \cite[Theorem 1]{e:73} yields a convex norm-compact metric space
$K$ such that $\free{K}$ also fails the AP \cite[Corollary 5]{go:14}. However, if $M$
is a countable proper metric space, or a proper ultrametric space, then $\free{M}$ has the MAP \cite{d:14-proper}.
The Lipschitz-free space over the Urysohn space has the MAP too, according to
Fonf and Wojtaszczyk \cite[Theorem 2.1]{fw:08}. Lately in \cite[Theorem 1]{cd:14}, C\'uth and Doucha
even built monotone Schauder bases in Lipschitz-free spaces over separable ultrametric spaces.

Let us focus now on the spaces $\R^N$ and their subsets. To prove the aforementioned equivalence between
$X$ having the $\lambda$-BAP and $\free{X}$ having the $\lambda$-BAP, Godefroy and Kalton
first show that $\free{\R^N}$ has the MAP with respect to any norm on $\R^N$ \cite[Proposition 5.1]{gk:03}.
In fact, $\free{\R^N}$ has a finite-dimensional Schauder decomposition \cite{bm:12}, which is monotone
when considered with respect to the $\ell_1$-norm \cite[Theorem 3.1]{lp:13}. This result was extended
in \cite[Theorem 3.1]{hp:14}, where Schauder bases of $\free{\ell_1^N}$ and $\free{\ell_1}$ were found.
In \cite[Proposition 2.3]{lp:13} it is shown that there is a universal constant $C$, such that $\free{M}$,
where $M\subseteq\R^N$ is arbitrary, has the $C\sqrt{N}$-BAP with respect to the Euclidean norm on $\R^N$.
One of the applications of Kaufmann's main result asserts that if $M\subseteq \R^N$ has non-empty
interior, then $\free{M}$ is linearly isomorphic to $\free{\R^N}$ \cite[Corollary 3.5]{kauf:14}.
Consequently, when combined with \cite[Theorem 3.1]{hp:14}, such a $\free{M}$ admits a Schauder basis
(having a basis constant which depends on the dimension $N$, as far as the present authors are aware). 

Our aim is to show that for certain subsets $M \subseteq \R^N$, the space $\mathcal{F}(M)$ has the MAP
with respect to any norm on $\R^N$. The following theorem is our main result.

\begin{thm}\label{thm:MAP}
Let $N \geqslant 1$ and consider $\R^N$ equipped with some norm $\ndot$. Let a compact set $M \subseteq \R^N$ have the property that given $\xi>0$, there
exists a set $\hat{M}\subseteq \R^N$ and a Lipschitz map $\map{\Psi}{\hat{M}}{M}$, such that $M \subseteq \intr(\hat{M})$, $\Lip(\Psi)\leqslant 1 + \xi$ and $\n{x-\Psi(x)} \leqslant \xi$ for all $x \in \hat{M}$. Then the Lipschitz-free space $\free{M}$ has the MAP.
\end{thm}

The proof of this theorem can be found in Section \ref{s:MAP} and relies on statements from Section \ref{s:preparatory}. The methods we use depend on the geometry of $M$ and, in particular,
that of its boundary $\bdy M$. In Section \ref{s:ldc-sets} we establish a sufficient condition on $\bdy M$
for $M$ to satisfy the hypotheses of Theorem \ref{thm:MAP} and thus for $\free{M}$ to admit the MAP with respect to any norm on $\R^N$.

Theorem \ref{thm:MAP} yields the following corollary, which we prove after Corollary \ref{cor:MAP}
below.

\begin{cor}\label{cor:finite-dim-convex}
Let $N \geqslant 1$ and let $M\subseteq \R^N$ be compact and convex. Then $\free{M}$ has the MAP
with respect to any norm on $\R^N$.
\end{cor}

The reader should compare this result to \cite[Corollary 5]{go:14} mentioned above, which asserts the existence
of a compact convex subset $M$ of an infinite-dimensional separable Banach space, whose Lipschitz-free space
$\free{M}$ fails the approximation property.

We do not know if $\free{M}$ has the MAP for all subsets $M \subseteq \R^N$.

\section{Locally downwards closed sets}\label{s:ldc-sets}

In this section we introduce a class of subsets of $\R^N$ satisfying the assumptions of Theorem \ref{thm:MAP}. Given $M \subseteq \R^N$, let $\intr(M)$ denote the interior of $M$.

\begin{defn}\label{def:downwards-closed}
Let $N\geqslant 1$ and $M \subseteq \R^N$. Given open $U \subseteq \R^N$ and $u \in \R^N$,
we shall say that $M$ is {\em downwards closed relative to $U$ and $u$} if it is closed and $y-tu \in \intr(M)$
whenever $y\in U\cap M$, $t>0$ and $y-tu \in U$. In addition, we will say that $M$
is {\em locally downwards closed} if, for every $x\in \R^N$, there is an open set $U \ni x$
and a vector $u\neq 0$, such that $M$ is downwards closed relative to $U$ and $u$.
\end{defn}

It is clear that this notion does not depend on the choice of norm on $\R^N$.
To test a set to see if it is locally downwards closed, it is only necessary to check
the condition at points of the boundary $\bdy M$:\ if $x \in \intr(M)$ or $x \in \R^N\setminus M$, then
$M$ is downwards closed with respect to $\intr(M)$ and
$\R^N\setminus M$, respectively, and any non-zero $u$. Thus, local downwards closure is a
regularity condition on $\bdy M$. It is designed to mimic the notion that, locally, the
boundary is the graph of a continuous function (subject to a suitable change of coordinates), without
having to mention any functions in the definition.

Certainly, any closed convex subset of $\R^N$ having non-empty interior is locally
downwards closed, as the next proposition shows. Given $x,y \in \R^N$ and $s>0$, let $[x,y]$ denote the
straight line segment between $x$ and $y$ and let $B(x,s)$  and $U(x,s)$ be the closed and open balls in
$\R^N$ having centre $x$ and radius $s$ with respect to the Euclidean norm $\pndot{2}$, respectively.

\begin{prop}\label{fat-star-shaped}
Let $M \subseteq \R^N$, $N\geqslant 1$, be closed and imagine that the set
\[
M_0 \;=\; \set{x \in M}{\text{$[x,y] \subseteq M$ for all $y \in M$}},
\]
contains an interior point $w$. Then $M$ is locally downwards closed.
\end{prop}

\begin{proof}\phantom\qedhere
Fix $w$ and $r>0$ such that $U(w,r)\subseteq M_0$. Given $x \in M\setminus \{w\}$,
set $u=x-w$ and $r'=\min\{r,\frac{1}{2}\pn{u}{2}\}$. Given $y\in U(x,r')\cap M$,
set $s=r'-\pn{y-x}{2}>0$. Evidently, $U(y+w-x,s)\subseteq U(w,r)\subseteq M_0$, and thus
$\conv(U(y+w-x,s) \cup \{y\}) \subseteq M$. In particular, if $0< t \leqslant 1$,
then $y-tu = y+t(w-x) \in \intr(M)$. If $t> 1$ then $y-tu \not\in U(x,r')$, because
\[
\pn{x-(y-tu)}{2} \;\geqslant\; t\pn{u}{2}-\pn{y-x}{2} \;>\; 2r'-r' \;=\; r'. \tag*{\qedsymbol}
\]
\end{proof}


In Proposition \ref{p:enlarging} below we show that a compact locally downwards
closed set $M\subseteq\R^N$ satisfies the hypotheses of Theorem \ref{thm:MAP}.

\begin{prop}\label{p:enlarging}
Let $M\subseteq\R^N$, $N\geqslant 1$, be a compact and locally downwards closed set and let $\xi>0$. Then for any norm $\ndot$ on $\R^N$ there
exists a set $\hat{M}\subseteq \R^N$ and a Lipschitz map $\map{\Psi}{\hat{M}}{M}$, such that $M \subseteq \intr(\hat{M})$, $\Lip(\Psi)\leqslant 1 + \xi$ and $\n{x-\Psi(x)} \leqslant \xi$ for all $x \in \hat{M}$.
\end{prop}

Most of Section \ref{s:ldc-sets} is concerned with proving Proposition \ref{p:enlarging}.
The next corollary is obtained as an easy consequence.

\begin{cor}\label{cor:MAP}
Let $M\subseteq \R^N$ be a compact and locally downwards closed set. Then $\free{M}$ has the MAP with respect to any norm on $\R^N$.
\end{cor}
\begin{proof}
The result follows directly from Theorem \ref{thm:MAP} and Proposition \ref{p:enlarging}.
\end{proof}

Corollary \ref{cor:MAP} allows us to prove Corollary \ref{cor:finite-dim-convex}.

\begin{proof}[Proof of Corollary \ref{cor:finite-dim-convex}]
Without loss of generality we can assume that $0 \in M$. If $X=\lspan{M}$, then $M$ has non-empty
interior relative to $X$. Working now in $X$, the result follows from Proposition
\ref{fat-star-shaped} and Corollary \ref{cor:MAP}.
\end{proof}

To prove Proposition \ref{p:enlarging} we first state and prove a few lemmas. Hereafter, we shall fix $N \geqslant 1$ and some norm $\ndot$ on $\R^N$.
Unless otherwise stated, all Lipschitz constants are taken with respect to $\ndot$. The symbol $\pndot{p}$, $p \in [1,\infty]$, stands for the $\ell_p$-norm on $\R^N$. We will have need of a constant $K>0$ satisfying
$\frac{1}{K}\ndot\leqslant\pndot{1},\pndot{2} \leqslant K\ndot$.

\begin{lem}\label{no-conv} Let $U \subseteq \R^N$ be a convex open set, $k\geqslant 1$ and let $M\subseteq \R^N$ be
downwards closed relative to $U$ and $u_1,\dots,u_k \in \R^N$. Fix $x \in U\setminus \intr(M)$
and $t_i > 0$, $1\leqslant i\leqslant k$, such that $x_i = x + t_i u_i \in U\setminus M$.
Then $M \cap \conv(x_1,\dots,x_k)$ is empty.
\end{lem}

\begin{proof} We proceed by induction on $k$. Let $y=\sum_{i=1}^k \lambda_i x_i
= x + \sum_{i=1}^k \lambda_i t_i u_i \in \conv(x_1,\dots,x_k)\subseteq U$, where
$\lambda_i \geqslant 0$ and $\sum_{i=1}^k \lambda_i=1$. If $k=1$ then
$y=x_1\not\in M$. Now suppose that $k>1$ and that the statement holds for $k-1$. We may assume that $\lambda_{k-1},\lambda_k>0$. Consider
\[
z \;=\; y - \lambda_k t_k u_k \;=\; \bigg(\sum_{i=1}^{k-2} \lambda_i x_i\bigg) + \left(\lambda_{k-1}+\lambda_k\right)\tilde x
\in \conv(x_1,\dots,x_{k-2},\tilde x),
\]
where $\tilde x =x+\lambda_{k-1}\left(\lambda_{k-1}+\lambda_k\right)^{-1}t_{k-1}u_{k-1}$. From the convexity of $U$ and downwards closure it follows that $\tilde x\in U\setminus M$. So, by inductive hypothesis, $z\notin M$. Hence $y\not\in M$, again by downwards closure.
\end{proof}

In the construction of $\hat M$ and $\Psi$ we will make use of a few auxiliary functions
that are perturbations of the identity, both in a Lipschitz and uniform sense.
Fix $\theta>1$, $x \in \R^N$, $r>0$ and $u \in \R^N$, $\pn{u}{2}=1$. We define the map
$\map{T_{\theta,x,r,u}}{\R^N}{\R^N}$ by
\begin{equation}\label{eq:T}
T_{\theta,x,r,u}(y) \;=\; y + (\theta-1)((y-x)\cdot u+r)u,
\end{equation}
where $\cdot$ denotes the scalar product.
For a map $\map{T}{E}{\R^N}$, where $E\subseteq\R^N$, we set
\[
\pn{T}{\infty} \;=\; \sup\set{\n{T(y)}}{y \in E}.
\]

\begin{lem}\label{estimate}
Let $\theta$, $x$, $r$, $u$ and $\;T_{\theta,x,r,u}$ be as above. Imagine that $E\subseteq \R^N$ is bounded, and set $P=\sup\set{\n{y-x}}{y \in E}<\infty$. Then $\Lip(T_{\theta,x,r,u}-I) \leqslant K^2(\theta-1)$ and $\pn{\left(T_{\theta,x,r,u}-I\right)\res{E}}{\infty}\leqslant K(KP+r)(\theta-1)$.
\end{lem}

\begin{proof}\phantom\qedhere
Let $y,z \in \R^N$. Then
\begin{align*}
\n{(T_{\theta,x,r,u}-I)z-(T_{\theta,x,r,u}-I)y} &\;=\; (\theta-1)|((z-x)\cdot u+r)-((y-x)\cdot u+r)|\n{u}\\
&\;\leqslant\; K(\theta-1)|(z-y)\cdot u|\\
&\;\leqslant\; K(\theta-1)\pn{z-y}{2}\pn{u}{2}\\
&\;\leqslant\; K^2(\theta-1)\n{z-y}.
\end{align*}
If $y \in E$ then
\begin{align*}
\n{(T_{\theta,x,r,u}-I)y} &\;=\; (\theta-1)|(y-x)\cdot u+r|\n{u}\\
&\;\leqslant\; K(\theta-1)(\pn{y-x}{2}\pn{u}{2} + r)\\
&\;\leqslant\; K(KP+r)(\theta-1). \tag*{\qedsymbol}
\end{align*}
\end{proof}

Next, we require a lemma about using partitions of unity to glue together
Lipschitz functions. If these functions are sufficiently close to the identity
map, both in a Lipschitz and uniform sense, then the resulting map is close
in both senses as well.

\begin{lem}\label{partition-of-unity}
Let $\xi>0$, $U$ an open subset of a normed space $X$, $(U_j)_{j=1}^k$ an open cover of
$U$, a partition of unity $\map{f_j}{U}{[0,1]}$, $1\leqslant j \leqslant k$,
subordinated to the cover and consisting of $H$-Lipschitz functions,
and functions $\map{\psi_j}{U}{X}$ such that
\[
\Lip(\psi_j-I),\, \pn{\psi_j-I}{\infty} \;\leqslant\; \xi,
\]
$1\leqslant j \leqslant k$. Then the function $\map{\psi}{U}{X}$, defined by
\begin{equation}\label{eq:glued-function}
\psi(x) \;=\; \sum_{j=1}^k f_j(x)\psi_j(x),
\end{equation}
satisfies $\Lip(\psi-I)\leqslant (1+Hk)\xi$ and $\pn{\psi-I}{\infty}\leqslant\xi$.
\end{lem}

\begin{proof}
Let $x,y \in U$. Then
\begin{align*}
\n{(\psi-I)y-(\psi-I)x} \;=\;& \n{\sum_{j=1}^k f_j(y)((\psi_j(y) - y)-(\psi_j(x)-x))
+(f_j(y)-f_j(x))\psi_j(x)}\\
\;\leqslant\;& \sum_{j=1}^k f_j(y)\n{(\psi_j(y) - y)-(\psi_j(x)-x)}\\
& + \n{\sum_{j=1}^k (f_j(y)-f_j(x))(\psi_j(x)-x)}\\
\;\leqslant\;& \xi\n{y-x} + \sum_{j=1}^k H\n{y-x}\xi \;=\; (1+Hk)\xi\n{y-x}.
\end{align*}
The other inequality follows easily.
\end{proof}

Of course, if $\Lip(\psi-I)\leqslant\xi<1$, then $\Lip(\psi)\leqslant 1+\xi$, and $\psi^{-1}$ exists
and satisfies $\Lip(\psi^{-1})\leqslant 1/(1-\xi)$, because
\begin{align*}
\n{\psi(y)-\psi(x)} &\;=\; \n{y-x-((\psi(x)-x)-(\psi(y)-y))}\\
&\;\geqslant\; \n{y-x}-\n{(\psi(x)-x)-(\psi(y)-y)}\\
&\;\geqslant\; (1-\xi)\n{y-x}.
\end{align*}

Now we are in a position to prove Proposition \ref{p:enlarging}.

\begin{proof}[Proof of Proposition \ref{p:enlarging}]\phantom\qedhere
Initially, we make the assumption that $M$ is connected, in addition to being compact
and locally downwards closed. Once we have dealt with the
connected case, we show how this assumption can be removed.

For each $x \in \bdy M$, let $r_x\in(0,1)$ and $u_x \in \R^N$, $\pn{u_x}{2}=1$, such that
$M$ is downwards closed with respect to $U(x,2r_x)$ and $u_x$. Let $x_1,\dots,x_k \in \bdy M$
such that $(U(x_i,r_i))_{i=1}^k$ is a cover of $\bdy M$, where $r_i=r_{x_i}$.
Set $U_i=U(x_i,r_i)$, $1\leqslant i \leqslant k$, and $U_{k+1}=\intr(M)$.
Define $U=\bigcup_{i=1}^{k+1} U_i \supseteq M$, and let $\map{f_i}{U}{\R^N}$,
$1\leqslant i \leqslant k+1$, be a partition of unity subordinated to $(U_i)_{i=1}^{k+1}$,
such that each $f_i$ is $H$-Lipschitz, for some large enough $H$.

Select $w\in\intr(M)$ and $\ep\in(0,\min\{1,\xi\})$ such that $B(w,\ep)\subseteq M$. Let
\[
s=\min\set{r_j/2r_i}{1\leqslant i,j \leqslant k},
\]
fix
\begin{equation}\label{eq:small-perturb-1}
{\ts \theta \;=\; \min\{1 + \frac{1}{2}K^{-2}(1+(k+1)H)^{-1}\ep, 1+K^{-1}(K\diam(U)+1)^{-1}\ep, 1+s \}},
\end{equation}
let $\psi_i=T_{\theta,x_i,r_i,u_i}$, $1\leqslant i\leqslant k$, where $u_i=u_{x_i}$ and $T_{\theta,x_i,r_i,u_i}$ is as in (\ref{eq:T}), and set $\psi_{k+1}=I$.
From Lemma \ref{estimate} and the fact that $r_i \leqslant 1$ for all $i$, we know that
\begin{equation}\label{eq:small-perturb-2}
\Lip(\psi_i-I) \;\leqslant\; K^2(\theta-1) \quad\text{and}\quad
\pn{(\psi_i-I)\res{U}}{\infty} \;\leqslant\; K(K\diam(U)+1)(\theta-1).
\end{equation}

If $\map{\psi}{U}{\R^N}$ is the map defined in (\ref{eq:glued-function}) (in Lemma \ref{partition-of-unity}), then
(\ref{eq:small-perturb-1}) and (\ref{eq:small-perturb-2}) yield
\[
\Lip(\psi-I) \;\leqslant\; K^2(1+(k+1)H)(\theta-1) \;\leqslant\; {\ts \frac{1}{2}}\ep \;<\; 1,
\]
and
\[
\pn{\psi-I}{\infty} \;\leqslant\; K(K\diam(U)+1)(\theta-1) \;\leqslant\; \ep.
\]
From above, we know therefore that $\psi^{-1}$ exists on $\psi(U)$,
\begin{equation}\label{eq:small-perturb-3}
\Lip(\psi^{-1}) \;\leqslant\; \frac{1}{1-{\ts \frac{1}{2}\ep}} \;\leqslant\; 1+\ep \;\leqslant\; 1+\xi \quad\text{and}\quad
\pn{\psi^{-1}-I}{\infty} \;\leqslant\; \ep \;\leqslant\; \xi.
\end{equation}

According to Brouwer's Theorem of Invariance of Domain, $\psi(U)$ is open in $\R^N$.
This implies that $\intr(\psi(M))$ is the same, relative to both $\psi(U)$ and $\R^N$,
and likewise for $\bdy\psi(M)$, so we can use the terms without fear of ambiguity
(of course, the same applies to $\intr(M)$ and $\bdy M$, relative to $U$ and $\R^N$).
We would like to show that $M \subseteq \intr(\psi(M))$.

First, we show that $M \cap \bdy\psi(M)$ is empty. Since $\psi$ is a homeomorphism of $U$ onto $\psi(U)$, we know that $\bdy\psi(M)=\psi(\bdy M)$.
Let $x \in \bdy M$ and set $I=\set{i\leqslant k+1}{x\in U_i}$. Of course, $I\subseteq\{1,\dots,k\}$
because $U_{k+1}=\intr(M)$. It follows that
\[
\psi(x) \;=\; \sum_{i=1}^{k+1}f_i(x)\psi_i(x) \;=\; \sum_{i \in I} f_i(x)\psi_i(x),
\]
where $\sum_{i\in I} f_i(x)=1$ and
\[
\psi_i(x)\;=\; x+ (\theta-1)((x-x_i)\cdot u_i+r_i)u_i.
\]
Given $i\in I$, we have $0<(x-x_i)\cdot u_i+r_i < 2r_i$, because
$\pn{x-x_i}{2}<r_i$. Consequently,
\begin{equation}\label{eq:small-perturb-4}
0\;<\;\pn{\psi_i(x)-x}{2} \;\leqslant\; 2r_i(\theta-1) \;\leqslant\; 2r_i s \;\leqslant\; r_j,
\end{equation}
whenever $i\in I$ and $1\leqslant j \leqslant k$, by (\ref{eq:small-perturb-1}).

Set $V=\bigcap_{i\in I} U(x_i,2r_i)$. Of course, $x \in \bigcap_{i\in I} U(x_i,r_i)\subseteq V$, and by (\ref{eq:small-perturb-4}),
$\psi_i(x) \in V$ whenever $i \in I$ as well. Since $M$ is downwards closed relative to $U(x_i,2r_i)$
and $u_i$, we must have $\psi_i(x) \notin M$, because $x\not\in \intr(M)$.
Thus, from Lemma \ref{no-conv}, we see that
\[
\psi(x) \;=\; \sum_{i \in I} f_i(x)\psi_i(x) \;\notin\; M.
\]
In particular, $M \cap \bdy\psi(M)=M \cap \psi(\bdy M)$ is empty.

Since $M \cap \bdy\psi(M)$ is empty, we can write $M$ as the union of two disjoint sets
$M \cap \intr(\psi(M))$ and $M\setminus\psi(M)$, which are both open in $M$. Since $\psi(w)\in B(w,\ep)\subseteq M$, we have $\psi(w) \in \psi(\intr(M)) \cap M= \intr(\psi(M)) \cap M$. Therefore,
by the connectedness of $M$, we know that $M=M \cap \intr\psi(M)\subseteq \intr\psi(M)$,
as claimed. We complete the proof in the connected case by setting $\hat{M}=\psi(M)$ and
$\Psi=\psi^{-1}$, and considering (\ref{eq:small-perturb-3}).

We approach the general case by showing that a compact and locally downwards closed set
decomposes into finitely many connected components, each one locally downwards closed.
Then we apply what we have done above to each component and glue the results together.

To prove that $M$ has just finitely many connected components, we begin by showing
that if $x\in M$, then there exists an open set $V\ni x$ such that
$M\cap V$ is connected. Indeed, pick an open Euclidean ball $U$ having centre $x$ and $u\neq 0$
such that $M$ is downwards closed relative to $U$ and $u$, fix any $t>0$ such that $x-tu \in U$,
and then fix $r>0$ such that $U(x-tu,r)\subseteq M\cap U$, which exists by virtue of downwards
closure. We claim that $M\cap V$ is path-connected, where $V$ is the open set
\[
V \;=\; \set{y\in \R^N}{\pn{y-(x-su)}{2} < r \text{ for some }s \in [0,t]} \;\subseteq \; U.
\]
Indeed, if $y \in M$ and $\pn{y-(x-su)}{2} < r$ for some $s \in [0,t]$, then local downwards
closure guarantees that $[y,y-(t-s)u]\subseteq M\cap V$, and as $y-(t-s)u \in U(x-tu,r) \subseteq M\cap V$,
path-connectivity is evident.

Now imagine, for a contradiction, that $M$ possesses
infinitely many connected components. Extract a sequence $(x_n)\subseteq M$ such that each $x_n$
belongs to a different component, and let $x\in M$ be a limit point of this sequence. From
above, there exists an open set $V\ni x$ such that $M\cap V$ is connected, however, this contradicts
that fact that $M\cap V$ must contain infinitely many of the $x_n$, each belonging to different
components of $M$.

Thus $M=\bigcup_{i=1}^p M_i$, where the $M_i$ denote the connected components of $M$.
Of course, each $M_i$ is compact and open in $M$, so each one is itself locally downwards
closed. Select $\alpha \in (0,1)$ with the property that $\n{x-y}\geqslant\alpha$
whenever $x$ and $y$ lie in distinct components. Given $\xi>0$, set $\xi' =
\min\{\frac{1}{2}\xi\alpha,\frac{1}{4}\alpha\}$. Using the result above for connected sets,
take $\hat{M}_i \subseteq \R^N$ and maps $\map{\Psi_i}{\hat{M}_i}{M_i}$, such that
$M_i \subseteq \intr(\hat{M}_i)$, $\Lip(\Psi_i)\leqslant 1 + \xi'$ and $\n{x-\Psi_i(x)}
\leqslant \xi'$ whenever $x \in \hat{M}_i$ and $1\leqslant i\leqslant p$. The $\hat{M}_i$
are pairwise disjoint, because the existence of $x \in \hat{M}_i \cap\hat{M}_j$, $i\neq j$,
would imply that
\[
\alpha \;\leqslant\; \n{\Psi_i(x)-\Psi_j(x)} \;\leqslant\; \n{\Psi_i(x)-x}+\n{x-\Psi_j(x)}
\;\leqslant\; {\ts \frac{1}{2}}\alpha.
\]
Define $\hat{M}=\bigcup_{i=1}^p \hat M_i$ and $\map{\Psi}{\hat{M}}{M}$ by $\Psi(x)=\Psi_i(x)$
whenever $x\in\hat{M}_i$. Certainly, $\n{x-\Psi(x)}\leqslant \xi$. Moreover, given $x,y \in M$,
either they are in the same $\hat{M}_i$, giving $\n{\Psi(x)-\Psi(y)}=\n{\Psi_i(x)-\Psi_i(y)}
\leqslant (1+\xi)\n{x-y}$, or they are in distinct $\hat{M}_i$ and $\hat{M}_j$, respectively, whence
\[
\n{\Psi(x)-\Psi(y)} \;\leqslant\; \n{\Psi_i(x)-x} + \n{\Psi_j(y)-y} + \n{x-y} \;\leqslant\;
(1+\xi)\n{x-y}. \tag*{\qedsymbol}
\]
\end{proof}

\section{Preparatory lemmas}\label{s:preparatory}
In order to prove Theorem \ref{thm:MAP}, we will need to demonstrate the existence of certain operators
on $\free{M}$. However, we shall work mostly with dual operators on the dual space $\Lipz{M}$, because
in our opinion $\Lipz{M}$ is a more `concrete' space than $\free{M}$ and, as a consequence, the dual
operators can be defined and described more easily.

In this section, we prove three lemmas that will be used in the proof of the main result.
Lemma \ref{Lip-perturb} makes use of small Lipschitz perturbations of the identity to map Lipschitz functions
on $M$ to Lipschitz functions on a slightly enlarged set,
without changing the Lipschitz constants very much.
Lemma \ref{l:smoothing} concerns the convolution of Lipschitz functions to make them smooth, and Lemma
\ref{interpolate} addresses the problem of approximating certain smooth Lipschitz functions by `coordinatewise
affine interpolation', again without increasing the Lipschitz constants by very much.

Given $M \subseteq \R^N$ and $r>0$, we define the open set
\begin{equation}\label{eq:M(r)}
M(r) \;=\; \set{x \in \R^N}{d_2(x,\R^N\setminus M) > r},
\end{equation}
where $d_2(\cdot,E)$ denotes distance to the set $E$ with respect to the Euclidean norm $\pndot{2}$.

\begin{lem}\label{Lip-perturb}
Let a compact set $M \subseteq \R^N$ have the property that given $\xi>0$, there
exists a set $\hat{M}\subseteq \R^N$ and a Lipschitz map $\map{\Psi}{\hat{M}}{M}$, such that $M \subseteq \intr(\hat{M})$, $\Lip(\Psi)\leqslant 1 + \xi$ and $\n{x-\Psi(x)} \leqslant \xi$ for all $x \in \hat{M}$.
Then, given $\ep>0$, there is $\hat{M}\subseteq \R^N$ such that $M \subseteq \hat{M}(r)$ for some $r>0$, and there is a dual operator
$\map{Q}{\Lipz{M}}{\Lipz{\hat{M}}}$ (where $M$ and $\hat{M}$ share the same distinguished
point $x_0$), such that $\n{Q}\leqslant 1 + \ep$ and
\[
|f(x)-Qf(x)| \leqslant \ep\Lip(f),
\]
whenever $x \in M$.
\end{lem}

\begin{proof}\phantom\qedhere
Let $\ep>0$. Let $\xi \in (0,\frac{1}{2}\ep)$, take $\hat{M}$ and $\Psi$ from the hypotheses, and define $Qf = f \circ \Psi - f(\Psi(x_0))\ind{\hat M}$.
Evidently, $\map{Q}{\Lipz{M}}{\Lipz{\hat{M}}}$, $\n{Q} \leqslant 1 + \ep$, and $Q$ has predual $Q_*$ given by $Q_*\delta_x = \delta_{\Psi(x)}-\delta_{\Psi(x_0)}$, $x \in \hat M$. By compactness and the fact that $M \subseteq \intr(\hat{M})$,
there exists $r >0$ such that $M \subseteq \hat{M}(r)$. If $x \in M$ then we estimate
\begin{align*}
|f(x)-Qf(x)| &\;=\; |f(x)-f(\Psi(x)) + f(\Psi(x_0))- f(x_0)|\\
&\;\leqslant\; |f(x)-f(\Psi(x))| + |f(\Psi(x_0))- f(x_0)|\\
&\;\leqslant\; (\n{x-\Psi(x)} + \n{\Psi(x_0)-x_0})\Lip(f)\\
&\;\leqslant\; \ep\Lip(f). \tag*{\qedsymbol}
\end{align*}
\end{proof}

We move on to Lemma \ref{l:smoothing}. Following \cite[pp. 629]{E}, we define $\eta:\R^N\to[0,\infty)$ by
\[
\eta(x)=\begin{cases}
A\exp\left(\frac{1}{\pn{x}{2}^2-1}\right)&\textup{ if }\pn{x}{2}<1,\\
0&\textup{ if }\pn{x}{2}\geqslant 1,
\end{cases}
\]
where the constant $A>0$ is chosen so that $\lint{\R^N}{}{\eta(x)}{x}=1$. Next, for each $s>0$, we put
\[
\eta_s(x)=\frac{1}{s^N}\;\eta\left(\frac{x}{s}\right).
\]
Then the function $\eta_s$ lies in $C^\infty(\R^N)$ and satisfies $\lint{\R^N}{}{\eta_s(x)}{x}=1$ and $\supp(\eta_s)\subseteq B(0,s)$.

Consider a bounded set $M\subseteq \R^N$ having non-empty interior, and distinguished point $x_0\in \intr(M)$.
Fix $r>0$ small enough so that $x_0\in M(r)$, where $M(r)$ is as in (\ref{eq:M(r)}).
For a locally integrable map $\map{f}{M}{\R}$ and $x\in M(r)$, define
\[
f_r(x)=(\eta_r\star f)(x)=\lint{M}{}{\eta_r(x-y)f(y)}{y}=\lint{B(0,r)}{}{\eta_r(y)f(x-y)}{y}.
\]
Finally, set
\begin{equation}
\label{eq:S_r}
S_r(f)=f_r-f_r(x_0)\ind{M(r)},
\end{equation}
on $M(r)$.

Given a function $\map{g}{\R^N}{\R}$, we denote by $Dg(x)$ its total derivative at $x$, should it exist.
We shall regard $Dg(x)$ both as a functional on $\R^N$ and as an $n$-tuple in $\R^N$, via the usual
identification.

\begin{lem}
\label{l:smoothing}
In the above setting, the mapping $S_r$ is a dual operator from $\Lipz{M}$ to $\Lipz{M(r)}$ (where $M$ and $M(r)$ have the same distinguished point $x_0$) and satisfies $\n{S_r}\leqslant 1$ and $S_r(\Lipz{M})\subseteq
C^\infty(M(r))$. Moreover, for every $\ep>0$, there exists $\delta>0$ having
the property that for every $f\in\Lipz{M}$, $x\in M(r)$ and $h\in\R^N$,
$\n{h}\leqslant\delta$, such that $x+h \in M(r)$, we have
\begin{equation}
 \label{uniform differentiability}
\left|S_r(f)(x+h)-S_r(f)(x)-DS_r(f)(x)[h]\right|\leqslant \ep\Lip(f)\n{h}.
\end{equation}
Finally, for every $f\in\Lipz{M}$ and $x\in M(r)$,
\begin{equation}
\label{pw convergence}
 |S_r(f)(x)-f(x)|\leqslant 2\Lip(f)Kr,
\end{equation}
where $K$ is as in Section \ref{s:ldc-sets}.
\end{lem}

\begin{proof}
Let $f\in\Lipz{M}$. Obviously, $S_r(f)(x_0)=0$. For $x,y\in M(r)$, we have
\begin{align}
\label{norm of S_r}
 \nonumber
|S_r(f)(x)-S_r(f)(y)|&=|f_r(x)-f_r(y)|\\\nonumber
&=\left|\lint{B(0,r)}{}{\eta_r(z)(f(x-z)-f(y-z))}{z}\right|\\\nonumber
&\leqslant\lint{B(0,r)}{}{\eta_r(z)\Lip(f)\n{x-y}}{z}\\
&=\Lip(f)\n{x-y}.
\end{align}
So $S_r$ is a well-defined mapping from $\Lipz{M}$ to $\Lipz{M(r)}$. Furthermore, it is
clearly linear and, by (\ref{norm of S_r}), bounded with $\n{S_r}\leqslant 1$.

Since we can identify compactly
supported Borel measures on $M$ with elements of $\free{M}$, we can see that the predual operator
$(S_r)_*$ can be defined by writing $\mathrm{d}((S_r)_*\delta_x) = (\eta_r(x-y) - \eta_r(x_0-y))\mathrm{d}y$,
$x \in M(r)$.

%
The fact that $S_r(f)\in C^\infty(M(r))$ follows from \cite[Appendix C.4, Theorem 6$\,$(i)]{E}. For $x\in M(r)$ and $h\in\R^N$, let
\[
L(x)[h]=\lint{M}{}{D\eta_r(x-y)[h]f(y)}{y}=\lint{B(0,r)}{}{D\eta_r(y)[h]f(x-y)}{y}.
\]
Clearly $L(x)$ is a bounded linear functional on $\R^N$. We will show that $L(x)$ is the derivative of $S_r(f)$ at $x$ and that the differentiability is uniform in the sense of (\ref{uniform differentiability}). Indeed, let $h\in\R^N$ be such that $x+h\in M(r)$. Then, by \cite[Corollary 4.99]{HJ},
\begin{align*}
&\; |S_r(f)(x+h)-S_r(f)(x)-L(x)[h]|\\
=&\; |f_r(x+h)-f_r(x)-L(x)[h]|\\
=&\; \left|\lint{M}{}{\left(\eta_r(x+h-y)-\eta_r(x-y)-D\eta_r(x-y)[h]\right)f(y)}{y}\right|\\
\leqslant &\; \lint{M}{}{\left|\left(\eta_r(x+h-y)-\eta_r(x-y)-D\eta_r(x-y)[h]\right)f(y)\right|}{y}\\
\leqslant &\; A(M)\Lip(f)\omega_{D\eta_r}(\n{h})\n{h},
\end{align*}
where $A(M)>0$ is a constant depending only on $M$, and $\omega_{D\eta_r}$ is the
modulus of continuity of $D\eta_r$. Hence $DS_r(f)(x)=L(x)$ and (\ref{uniform differentiability})
holds whenever $\delta>0$ is chosen to satisfy $\omega_{D\eta_r}(\delta)\leqslant\frac{\ep}{A(M)}$.

To conclude, for any $x\in M(r)$,
\begin{align*}
 |f_r(x)-f(x)|&=\left|\lint{B(0,r)}{}{\eta_r(y)(f(x-y)-f(x))}{y}\right|\\
&\leqslant \lint{B(0,r)}{}{\eta_r(y)|f(x-y)-f(x)|}{y}\\
&\leqslant \Lip(f)Kr.
\end{align*}
Thus, for every $x\in M(r)$, we obtain
\[
|S_r(f)(x)-f(x)|=|f_r(x)-f_r(x_0)-f(x)+f(x_0)| \leqslant 2\Lip(f)Kr,
\]
which yields (\ref{pw convergence}).
\end{proof}

Finally we address Lemma \ref{interpolate} and the approximation of smooth Lipschitz functions
by coordinatewise affine functions.
Fix $w \in \R^N$.
We define a closed hypercube $C\subseteq\R^N$ having edge length $\delta>0$ and vertices
$v_\gamma \in \R^N$, $\gamma \in \{0,1\}^N$, given by $v_\gamma
= w + \delta \gamma$. We will write $V_C$ for the set of all vertices of $C$. 

Imagine that $f$ is a real-valued function whose domain of definition includes the set $V_C$.
We define the {\em interpolation function} $\Lambda(f,C)$ on $\R^N$ by
\begin{equation}\label{def:lambda}
\Lambda(f,C)(x)=\sum_{\gamma\in\{0,1\}^N}\left(\prod_{i=1}^N\left(1-\gamma_i+(-1)^{\gamma_i+1}\frac{x_i-w_i}{\delta}\right)\right)f(v_\gamma).
\end{equation}

This is the same $\Lambda(f,C)$ as defined in \cite[Section 3.1]{lp:13} and \cite[Section 2.1]{hp:14},
except that in those cases the function is defined inductively, rather than by means of an explicit formula.
This function is coordinatewise affine, i.e., $t \mapsto
\Lambda(f,C)(x_1,\dots,x_{i-1},t,x_{i+1},\dots,x_N)$ is affine whenever $1 \leqslant i \leqslant N$.
Of course, $\Lambda(f,C)$ agrees with $f$ on the vertices of $C$ and, moreover, it is the only
coordinatewise affine function to do so. In the following lemma, we estimate
the Lipschitz constant of $\Lambda(f,C)$ on $C$, given a certain uniform differentiability assumption on $f$.
Below, the sequence $(e_i)_{i=1}^N$ denotes the standard unit vector basis of $\R^N$.

\begin{lem}\label{interpolate}
Let $\ep>0$, $U\subseteq \R^N$ be open and let $\map{f}{U\subseteq \R^N}{\R}$ be a differentiable Lipschitz function. Moreover, suppose there exists $\delta>0$ such that for each $x\in U$ and each $h\in\R^N$ with $\n{h}\leqslant K\delta$ and $x+h\in U$, we have
\begin{equation}\label{unif2}
|f(x+h)-f(x)-Df(x)[h]| \;\leqslant\; \ep\n{h}.
\end{equation}
Then, given a hypercube $C\subseteq U$ as above, having edge length $\delta$, and $\Lambda(f,C)$ as in $(\ref{def:lambda})$, we have
\begin{equation}
 \label{eq:lipschitz constant on a cube}
\Lip(\Lambda(f,C)\res{C})\;\leqslant\; K^2\ep+\Lip(f).
\end{equation}

In addition, for every $x\in C$,
\begin{equation}
\label{eq:uniform approximation on cubes}
|\Lambda(f,C)(x)-f(x)|\leqslant \sqrt NK\Lip(f)\delta.
\end{equation}
\end{lem}

\begin{proof}\phantom\qedhere
Fix a hypercube $C\subseteq U$ having edge length $\delta$ and
let $z\in \intr(C)$. We take $j\in\{1,\dots,N\}$ and compute the $j$-th partial derivative
of the function $\Lambda(f,C)$ at $z$. By (\ref{def:lambda}), we have
\begin{align*}
\frac{\partial\Lambda(f,C)}{\partial x_j}(z)&=\sum_{\gamma\in\{0,1\}^N}\left(\prod_{\substack{i=1\\i\neq j}}^N\left(1-\gamma_i+(-1)^{\gamma_i+1}\frac{z_i-w_i}{\delta}\right)\right)\frac{(-1)^{\gamma_j+1}}{\delta}f(v_\gamma)\\
&=\sum_{\substack{\gamma\in\{0,1\}^N\\\gamma_j=0}}\left(\prod_{\substack{i=1\\i\neq j}}^N\left(1-\gamma_i+(-1)^{\gamma_i+1}\frac{z_i-w_i}{\delta}\right)\right)\frac{f(v_\gamma+\delta e_j)-f(v_\gamma)}{\delta}.
\end{align*}
Now, for any $\gamma\in\{0,1\}^N$ such that $\gamma_j=0$, we can write
\begin{align*}
\frac{f(v_\gamma+\delta e_j)-f(v_\gamma)}{\delta}&=\left(1-\frac{z_j-w_j}{\delta}\right)\frac{f(v_\gamma+\delta e_j)-f(v_\gamma)}{\delta}\\
&\phantom{=}+\frac{z_j-w_j}{\delta}\frac{f(v_\gamma)-f(v_\gamma+\delta e_j)}{-\delta}\\
&=\left(1-\gamma_j+(-1)^{\gamma_j+1}\frac{z_j-w_j}{\delta}\right)\frac{f(v_\gamma+(-1)^{\gamma_j}\delta e_j)-f(v_\gamma)}{(-1)^{\gamma_j}\delta}\\
&\phantom{=}+\left(1-\tilde\gamma_j+(-1)^{\tilde\gamma_j+1}\frac{z_j-w_j}{\delta}\right)\frac{f(v_{\tilde\gamma}+(-1)^{\tilde\gamma_j}\delta e_j)-f(v_{\tilde\gamma})}{(-1)^{\tilde\gamma_j}\delta},
\end{align*}
where $\tilde\gamma\in\{0,1\}^N$ satisfies $\tilde\gamma_j=1$ and
$\tilde\gamma_i=\gamma_i$ for $i\neq j$. Hence
\[
\frac{\partial\Lambda(f,C)}{\partial x_j}(z)=\sum_{\gamma\in\{0,1\}^N}\left(\prod_{i=1}^N\left(1-\gamma_i+(-1)^{\gamma_i+1}\frac{z_i-w_i}{\delta}\right)\right)\frac{f(v_\gamma+(-1)^{\gamma_j}\delta e_j)-f(v_\gamma)}{(-1)^{\gamma_j}\delta}.
\]
So, for the total derivative of the function $\Lambda(f,C)$ at $z$ we obtain
\begin{align*}
D\Lambda(f,C)(z)&=\sum_{\gamma\in\{0,1\}^N}\Bigg(\left(\prod_{i=1}^N\left(1-\gamma_i+(-1)^{\gamma_i+1}\frac{z_i-w_i}{\delta}\right)\right)\\
&\phantom{=}\times\left(\frac{f(v_\gamma+(-1)^{\gamma_j}\delta e_j)-f(v_\gamma)}{(-1)^{\gamma_j}\delta}\right)_{j=1}^N\Bigg).
\end{align*}
Let $\tndot$ denote the dual of $\ndot$. For every $\gamma\in\{0,1\}^N$, we have
\begin{align*}
\bigg|\kern-.9pt\bigg|\kern-.9pt\bigg|\bigg(&\frac{f(v_\gamma+(-1)^{\gamma_j}\delta e_j)-f(v_\gamma)}{(-1)^{\gamma_j}\delta}\bigg)_{j=1}^N\bigg|\kern-.9pt\bigg|\kern-.9pt\bigg|\\
&\leqslant\bigg|\kern-.9pt\bigg|\kern-.9pt\bigg|\left(\frac{f(v_\gamma+(-1)^{\gamma_j}\delta e_j)-f(v_\gamma)-Df(v_\gamma)[(-1)^{\gamma_j}\delta e_j]}{(-1)^{\gamma_j}\delta}\right)_{j=1}^N
\bigg|\kern-.9pt\bigg|\kern-.9pt\bigg|+\tn{Df(v_\gamma)}\\
&\leqslant \frac{K}{\delta}\pn{\big(f(v_\gamma+(-1)^{\gamma_j}\delta e_j)-f(v_\gamma)-Df(v_\gamma)[(-1)^{\gamma_j}\delta e_j]\big)_{j=1}^N}{\infty}+\Lip(f).
\end{align*}
Thus, by (\ref{unif2}),
\begin{align*}
\bigg|\kern-.9pt\bigg|\kern-.9pt\bigg|\left(\frac{f(v_\gamma+(-1)^{\gamma_j}\delta e_j)-f(v_\gamma)}{(-1)^{\gamma_j}\delta}\right)_{j=1}^N\bigg|\kern-.9pt\bigg|\kern-.9pt\bigg| &\leqslant \frac{K}{\delta}\cdot\ep\n{(-1)^{\gamma_j}\delta e_j} + \Lip(f)\\
&\leqslant K^2\ep+\Lip(f).
\end{align*}
To conclude, since $D\Lambda(f,C)(z)$ lies in the convex hull of the set
\[
\set{\left(\frac{f(v_\gamma+(-1)^{\gamma_j}\delta e_j)-f(v_\gamma)}{(-1)^{\gamma_j}\delta}\right)_{j=1}^N} {\gamma\in\{0,1\}^N},\]
we have
\[
\tn{D\Lambda(f,C)(z)}\leqslant  K^2\ep+\Lip(f).
\]

Therefore $\Lip(\Lambda(f,C)\res{C})\leqslant K^2\ep+\Lip(f)$.

Finally, if $x\in C$, then
\begin{align*}
|\Lambda(f,C)(x)-f(x)|&= \left|\sum_{\gamma\in\{0,1\}^N}\left(\prod_{i=1}^N\left(1-\gamma_i+(-1)^{\gamma_i+1}\frac{x_i-w_i}{\delta}\right)\right)f(v_\gamma)-f(x)\right|\\
&\leqslant \sum_{\gamma\in\{0,1\}^N}\left(\prod_{i=1}^N\left(1-\gamma_i+(-1)^{\gamma_i+1}\frac{x_i-w_i}{\delta}\right)\right)\left|f(v_\gamma)-f(x)\right|\\
&\leqslant \sum_{\gamma\in\{0,1\}^N}\left(\prod_{i=1}^N\left(1-\gamma_i+(-1)^{\gamma_i+1}\frac{x_i-w_i}{\delta}\right)\right)\n{v_\gamma-x}\Lip(f).
\end{align*}
As $\n{v_\gamma-x}\leqslant K\pn{v_\gamma-x}{2} \leqslant \sqrt NK\delta$, we have
\[
|\Lambda(f,C)(x)-f(x)|\leqslant \sqrt NK\Lip(f)\delta. \tag*{\qedsymbol}
\]
\end{proof}

\begin{rem}
The interpolation result \cite[Lemma 3.2]{lp:13} follows quickly from the proof above.
In \cite{lp:13}, $\R^N$ is considered only with respect to $\pndot{1}$.
If $\ndot=\pndot{1}$, then for each $\gamma \in \{0,1\}^N$ we get
\begin{align*}
\bigg|\kern-.9pt\bigg|\kern-.9pt\bigg|\bigg(\frac{f(v_\gamma+(-1)^{\gamma_j}\delta e_j)-f(v_\gamma)}{(-1)^{\gamma_j}\delta}\bigg)_{j=1}^N\bigg|\kern-.9pt\bigg|\kern-.9pt\bigg|
&\;=\; \bigg\|\bigg(\frac{f(v_\gamma+(-1)^{\gamma_j}\delta e_j)-f(v_\gamma)}{(-1)^{\gamma_j}\delta}\bigg)_{j=1}^N\bigg\|_\infty\\
&\;\leqslant\; \Lip(f).
\end{align*}
Consequently, $\tn{D\Lambda(f,C)(z)}\leqslant \Lip(f)$ and $\Lip(\Lambda(f,C)\res{C})\leqslant\Lip(f)$.
\end{rem}

\section{The proof of Theorem \ref{thm:MAP}}\label{s:MAP}

The proof of Theorem \ref{thm:MAP} combines the processes spelled out in Lemmas \ref{Lip-perturb} -- \ref{interpolate}:
first, we approximate a given Lipschitz function by another Lipschitz function
on a slightly larger domain, then we apply a convolution to produce a smooth Lipschitz function,
and finally we approximate this smooth function by a number of locally coordinatewise affine
functions.

\begin{proof}[Proof of Theorem \ref{thm:MAP}]
We will build a sequence $(T_n)_{n=1}^{\infty}$ of finite-rank dual operators
on $\Lipz{M}$ such that $\n{T_n}\leqslant 1+n^{-1}$ for all $n\in\N$
and $T_n(f)(x) \to f(x)$ uniformly, simultaneously in $x \in M$ and $f\in B_{\Lipz{M}}$. Once we have
done this, we indicate why this yields the MAP at the end of the proof.

Let $n\in\N$. By Lemma \ref{Lip-perturb}, there is $\hat M_n\subseteq\R^N$ and
\begin{equation}
 \label{eq:r}
r_n \in \left(0,\frac{1}{n+1}\right),
\end{equation}
such that $M\subseteq \hat M_n((\frac{1}{2}K+2)r_n)$ (where $K$ is as in Section \ref{s:ldc-sets} and $\hat M_n((\frac{1}{2}K+2)r_n)$ is as in (\ref{eq:M(r)})), and there exists
a bounded linear operator $Q_n:\Lipz{M}\to\Lipz{\hat M_n}$ satisfying $\n{Q_n}\leqslant 1+(3n)^{-1}$
and
\begin{equation}\label{eq:Q}
|f(x)-Q_n(f)(x)| \leqslant n^{-1}\Lip(f),
\end{equation}
whenever $f\in\Lipz{M}$ and $x \in M$. Next, we press the smoothing operator $\map{S_{r_n}}{\Lipz{\hat M_n}}{\Lipz{\hat M_n(r_n)}}$
defined in (\ref{eq:S_r}) into service. By Lemma \ref{l:smoothing}, $S_{r_n}(\Lipz{\hat M_n})\subseteq
C^\infty(\hat M_n(r_n))$, and there exists
\begin{equation}
\label{eq:delta}
\delta_n\in\left(0,\frac{r_n}{2\sqrt NK(3n+1)}\right),
\end{equation}
such that for every $g\in\Lipz{\hat M_n}$, every $x\in \hat M_n(r_n)$ and every $h\in\R^N$,
$\n{h}\leqslant\delta_n$, satisfying $x+h\in \hat M_n(r_n)$, we have
\begin{equation}\label{eq:unif-main}
\left|S_{r_n}(g)(x+h)-S_{r_n}(g)(x)-DS_{r_n}(g)(x)[h]\right|\leqslant \frac{1}{3nK^2}\Lip(g)\n{h}.
\end{equation}
Consider the cover $\mathcal C_n$ of $\hat M_n(2r_n)$ by hypercubes of edge length $\delta_n$,
determined by the mesh $\mathcal Z_n=\{x_0+\delta_n\zeta, \zeta\in \Z^N\}$. In other words,
\[
\mathcal C_n=\set{C\subseteq \R^N}{C \text{ is a hypercube}, V_C=C\cap \mathcal Z_n \text{ and }
C\cap \hat M_n(2r_n)\neq\varnothing}
\]
(here we recall that $V_C$ is the set of all vertices of $C$).
According to (\ref{eq:delta}), we have $\bigcup \mathcal C_n\subseteq \hat M_n(r_n)$.
Define $V_n=\bigcup_{C\in\mathcal C_n}V_C$.

Given $f\in \Lipz{M}$ and $x\in M$, set
\[
T_n(f)(x)=\Lambda\left(S_{r_n}\left(Q_n(f)\right),C\right)(x),
\]
whenever $x\in C \in \mathcal{C}_n$. Observe that the definition of the
$\Lambda$ functions ensures that if $x \in C \cap C'$ and $C,C' \in \mathcal C_n$, then
\[
\Lambda(S_{r_n}(Q_n(f)),C)(x)=\Lambda(S_{r_n}(Q_n(f)),C')(x).
\]
Therefore $T_n$ is well-defined.

We will show that $T_n$ has the required properties. Fix $f\in B_{\Lipz{M}}$. To begin with,
\[
T_n(f)(x_0)=S_{r_n}\left(Q_n(f)\right)(x_0),
\]
because $x_0\in V_n$. Therefore $T_n(f)(x_0)=0$. Let $x,y\in M$. Recall that $\n{S_{r_n}}\leqslant 1$
and $\n{Q_n}\leqslant 1+(3n)^{-1}$. If $\n{x-y}\geqslant r_n$, then
\begin{align*}
&\;|T_n(f)(x)-T_n(f)(y)|\\
\leqslant&\; |S_{r_n}\left(Q_n(f)\right)(x)-S_{r_n}\left(Q_n(f)\right)(y)|
+2\sqrt NK(1+(3n)^{-1})\delta_n \tag*{by (\ref{eq:uniform approximation on cubes})}\\
\leqslant&\; |S_{r_n}\left(Q_n(f)\right)(x)-S_{r_n}\left(Q_n(f)\right)(y)|+(3n)^{-1}r_n \tag*{by (\ref{eq:delta})}\\
\leqslant&\; (1+(3n)^{-1})\n{x-y}+ (3n)^{-1}\n{x-y}\\
\leqslant&\; (1+n^{-1})\n{x-y}.
\end{align*}
On the other hand, if $\n{x-y}\leqslant r_n$, then $z \in [x,y]$ implies
$d_2(z,M)\leqslant d_2(z,\{x,y\})\leqslant \frac{1}{2}\pn{x-y}{2} \leqslant \frac{1}{2}Kr_n$, so
as $M\subseteq \hat{M}_n((\frac{1}{2}K+2)r_n)$, the line segment
$[x,y]$ lies entirely inside $\hat M_n(2r_n)\subseteq \bigcup \mathcal C_n \subseteq \hat M_n(r_n)$. Thus,
by partitioning $[x,y]$ with respect to the hypercubes through which it passes, the estimate of the
Lipschitz constants of the interpolation functions on hypercubes (\ref{eq:lipschitz constant on a cube}), which follows from (\ref{eq:unif-main}) by Lemma \ref{interpolate}, yields
\begin{align*}
|T_n(f)(x)-T_n(f)(y)|&\leqslant ((3n)^{-1}\Lip(Q_n(f))+\Lip(S_{r_n}(Q_n(f))))\n{x-y}\\
&\leqslant (1+(3n)^{-1})^2\n{x-y} \leqslant (1+n^{-1})\n{x-y}.
\end{align*}
Thus we conclude that $T_n$ is a well-defined mapping on $\Lipz{M}$. Moreover, it is
obviously a linear operator and $\n{T_n}\leqslant 1+n^{-1}$.

Given $x\in V_n$, denote by $\map{\phi_x}{\bigcup \mathcal{C}_n}{\R}$ the unique Lipschitz function that is coordinatewise affine
on each $C\in\mathcal C_n$ and satisfies $\phi_x\res{V_n}=\ind{\{x\}}\res{V_n}$. Since
$T_n(\Lipz{M})\subseteq \aspan\set{\phi_x\res{M}}{x\in V_n}$, the operator $T_n$ is of finite rank.

As stated in Lemmas \ref{Lip-perturb} and \ref{l:smoothing}, the operators $Q_n$ and $S_{r_n}$
are both dual operators. From the definition of the interpolation formula (\ref{def:lambda}),
it is easy to see that $T_n$ is also a dual operator.

Finally, by combining (\ref{eq:uniform approximation on cubes}), (\ref{pw convergence}) and
(\ref{eq:Q}), we get
\begin{align*}
|T_n(f)(x)-f(x)| \leqslant &\; |T_n(f)(x)-S_{r_n}(Q_n(f))(x)| + |S_{r_n}(Q_n(f))(x)-Q_n(f)(x)|\\
& + |Q_n(f)(x)-f(x)| \\
\leqslant&\; \sqrt NK(1+(3n)^{-1})\delta_n + 2K(1+(3n)^{-1})r_n + n^{-1},
\end{align*}
for $f\in B_{\Lipz{M}}$ and $x\in M$.
Then the choice of $\delta_n$ and $r_n$ (see (\ref{eq:delta}) and (\ref{eq:r}), respectively) gives
\begin{equation}\label{unif}
|T_n(f)(x)-f(x)| \leqslant \frac{1}{6n(n+1)} + (2K + 1)n^{-1}.
\end{equation}
Thus $T_n(f)(x)\to f(x)$ uniformly, both in $x \in M$ and $f \in B_{\Lipz{M}}$.

Now we explain why this means that $\free{M}$ has the MAP. Since $T_n$ is a dual operator, by
(\ref{unif}), the finite-rank predual operator $(T_n)_*$ on $\free{M}$ satisfies
\[
\n{(T_n)_*\delta_x - \delta_x} \to 0
\]
for all $x \in M$. Consequently, $(T_n)_*$ converges to the identity operator in the
strong operator topology. Since $\n{(T_n)_*} \to 1$, we deduce that $\free{M}$ has the MAP.
\end{proof}

\end{document}